\documentclass[10pt,oneside]{amsart}

\usepackage[
paper=a4paper,
headsep=20pt,text={136mm,205mm},centering,includehead
]{geometry}

\usepackage{amsmath,graphics,color}
\usepackage{amsfonts,amssymb}
\usepackage{epstopdf}
\usepackage{xypic,hyperref}

\iftrue
\makeatletter
\def\@settitle{%
  \vspace*{-8pt}
  \begin{flushleft}%
    \LARGE\bfseries
    \strut\@title\strut
  \end{flushleft}%
}
\def\@setauthors{%
  \begingroup
  \def\thanks{\protect\thanks@warning}%
  \trivlist
  \raggedright
  \large \@topsep27\p@\relax
  \advance\@topsep by -\baselineskip
  \item\relax
  \author@andify\authors
  \def\\{\protect\linebreak}%
  \authors
  \ifx\@empty\contribs
  \else
    ,\penalty-3 \space \@setcontribs
    \@closetoccontribs
  \fi
  \normalfont
  \endtrivlist
  \endgroup
}
\def\@setaddresses{\par
  \nobreak \begingroup
  \small\raggedright
  \def\author##1{\nobreak\addvspace\smallskipamount}%
  \def\\{\unskip, \ignorespaces}%
  \interlinepenalty\@M
  \def\address##1##2{\begingroup
    \par\addvspace\bigskipamount\noindent
    \@ifnotempty{##1}{(\ignorespaces##1\unskip) }%
    {\ignorespaces##2}\par\endgroup}%
  \def\curraddr##1##2{\begingroup
    \@ifnotempty{##2}{\nobreak\noindent\curraddrname
      \@ifnotempty{##1}{, \ignorespaces##1\unskip}\/:\space
      ##2\par}\endgroup}%
  \def\email##1##2{\begingroup
    \@ifnotempty{##2}{\nobreak\noindent E-mail address%
      \@ifnotempty{##1}{, \ignorespaces##1\unskip}\/:\space
      \ttfamily##2\par}\endgroup}%
  \def\urladdr##1##2{\begingroup
    \def~{\char`\~}%
    \@ifnotempty{##2}{\nobreak\noindent\urladdrname
      \@ifnotempty{##1}{, \ignorespaces##1\unskip}\/:\space
      \ttfamily##2\par}\endgroup}%
  \addresses
  \endgroup
  \global\let\addresses=\@empty
}
\def\@setabstracta{%
    \ifvoid\abstractbox
  \else
    \skip@17pt \advance\skip@-\lastskip
    \advance\skip@-\baselineskip \vskip\skip@
    \box\abstractbox
    \prevdepth\z@ 
    \vskip-15pt
  \fi
}
\renewenvironment{abstract}{%
  \ifx\maketitle\relax
    \ClassWarning{\@classname}{Abstract should precede
      \protect\maketitle\space in AMS document classes; reported}%
  \fi
  \global\setbox\abstractbox=\vtop \bgroup
    \normalfont\small
    \list{}{\labelwidth\z@
      \leftmargin0pc \rightmargin\leftmargin
      \listparindent\normalparindent \itemindent\z@
      \parsep\z@ \@plus\p@
      
    }%
    \item[\hskip\labelsep\bfseries\abstractname.]%
}{%
  \endlist\egroup
  \ifx\@setabstract\relax \@setabstracta \fi
}

\def\ps@headings{\ps@empty
  \def\@evenhead{%
    \setTrue{runhead}%
    \normalfont\scriptsize
    \rlap{\thepage}\hfill
    \def\thanks{\protect\thanks@warning}%
    \leftmark{}{}}%
  \def\@oddhead{%
    \setTrue{runhead}%
    \normalfont\scriptsize
    \def\thanks{\protect\thanks@warning}%
    \rightmark{}{}\hfill \llap{\thepage}}%
  \let\@mkboth\markboth
}\ps@headings

\def\section{\@startsection{section}{1}%
  \z@{-1.4\linespacing\@plus-.5\linespacing}{.8\linespacing}%
  {\normalfont\bfseries\Large}}
\def\subsection{\@startsection{subsection}{2}%
  \z@{-.8\linespacing\@plus-.3\linespacing}{.5\linespacing\@plus.2\linespacing}%
  {\normalfont\bfseries\large}}
\def\subsubsection{\@startsection{subsubsection}{3}%
  \z@{.7\linespacing\@plus.2\linespacing}{-1.5ex}%
  {\normalfont\bfseries}}
\def\@secnumfont{\bfseries}

\renewcommand\contentsnamefont{\bfseries}
\def\@starttoc#1#2{\begingroup
  \setTrue{#1}%
  \par\removelastskip\vskip\z@skip
  \@startsection{}\@M\z@{\linespacing\@plus\linespacing}%
    {.5\linespacing}{
      \contentsnamefont}{#2}%
  \ifx\contentsname#2%
  \else \addcontentsline{toc}{section}{#2}\fi
  \makeatletter
  \@input{\jobname.#1}%
  \if@filesw
    \@xp\newwrite\csname tf@#1\endcsname
    \immediate\@xp\openout\csname tf@#1\endcsname \jobname.#1\relax
  \fi
  \global\@nobreakfalse \endgroup
  \addvspace{32\p@\@plus14\p@}%
  \let\tableofcontents\relax
}
\def\contentsname{Contents}
\def\l@section{\@tocline{2}{.5ex}{0mm}{5pc}{}}
\def\l@subsection{\@tocline{2}{0pt}{2em}{5pc}{}}
\makeatother
\fi 

\def\to{\mathchoice{\longrightarrow}{\rightarrow}{\rightarrow}{\rightarrow}}
\makeatletter
\newcommand{\shortxra}[2][]{\ext@arrow 0359\rightarrowfill@{#1}{#2}}
\def\longrightarrowfill@{\arrowfill@\relbar\relbar\longrightarrow}
\newcommand{\longxra}[2][]{\ext@arrow 0359\longrightarrowfill@{#1}{#2}}

\newcommand{\tmfrac}[2]{\mbox{\large$\frac{#1}{#2}$}} 

\theoremstyle{plain}
\newtheorem*{theorem*}{Theorem}
\newtheorem*{lemma*} {Lemma}
\newtheorem*{corollary*} {Corollary}
\newtheorem*{proposition*} {Proposition}
\newtheorem{theorem}{Theorem}[section]
\newtheorem{lemma}[theorem]{Lemma}

\theoremstyle{remark}

\newtheorem*{claim}{Claim}

\theoremstyle{definition}

\def\be{\begin{equation}}
\def\ee{\end{equation}}
\def\co{\colon}
\def\RR{\mathcal{R}}

\def \K {\mathbf{K}}\def \Z {\mathbf{Z}}

\def\R{\mathbb{R}}
\def\eps{\epsilon}

\def\K{\mathbb{K}}

\def\Z{\mathbb{Z}}

\def\N{\mathbb{N}}
\def\part{\partial}

\def\bp{\begin{pmatrix}}
\def\sm{\setminus}\def\ep{\end{pmatrix}}
\def\bn{\begin{enumerate}}

\def\en{\end{enumerate}}\def\ba{\begin{array}}
\def\ea{\end{array}}

\def\fr12{\frac{1}{2}}

\def\ol{\overline}

\def\op{\operatorname}

\def\K{\mathbb{K}}

\def\cmtbf#1{} \def\cmt#1{}

\def\PP{\mathcal{P}}
\def\ZZ{\mathcal{Z}}

\def\QQ{\mathcal{Q}}
\def\XX{\mathcal{X}}
\def\YY{\mathcal{Y}}

\def\mfp{\mathfrak{P}}
\def\mfg{\mathfrak{G}}
\def\mfpsym{\mathfrak{P}^{\op{sym}}}
\def\mfgsym{\mathfrak{G}^{\op{sym}}}
\def\mfpnorm{\mathfrak{P}^{\op{norm}}}
\def\mfgnorm{\mathfrak{G}^{\op{norm}}}

\begin{document}

\title{The Grothendieck group of polytopes and norms}

\author{Jae Choon Cha}
\address{
  Department of Mathematics\\ POSTECH\\ Pohang 37673\\ Republic of Korea
  \linebreak
  School of Mathematics\\ Korea Institute for Advanced Study \\ Seoul 02455\\ Republic of Korea
}
\email{jccha@postech.ac.kr}

\author{Stefan Friedl}
\address{Fakult\"at f\"ur Mathematik\\ Universit\"at Regensburg\\ Germany}
\email{sfriedl@gmail.com}

\author{Florian Funke}
\address{Mathematisches Institut\\Universit\"at Bonn\\ Germany}
\email{ffunke@math.uni-bonn.de}

\def\subjclassname{\textup{2000} Mathematics Subject Classification}
\expandafter\let\csname subjclassname@1991\endcsname=\subjclassname \expandafter\let\csname
subjclassname@2000\endcsname=\subjclassname 

\begin{abstract} 
  Polytopes in $\R^n$ with integral vertices form a monoid under the
  Minkowski sum, and the Grothendieck construction gives rise to a
  group.  We show that every symmetric polytope is a norm in this
  group for every~$n$.
\end{abstract}
\maketitle

\section{Introduction}

In this paper we define a \emph{polytope} in $\R^n$ to be the convex
hull of a finite subset of $\R^n$.  If the finite subset lies in the
lattice $\Z^n$ in~$\R^n$, then we say that the polytope is
\emph{integral}.


We denote by $\mathfrak{P}(n)$ the set of all integral polytopes in~$\R^n$.
Given two polytopes $\PP$ and $\QQ$ in $\R^n$, the \emph{Minkowski sum
  of $\PP$ and $\QQ$} is defined to be the polytope
\[
  \PP+\QQ:=\{ p+q \mid p\in \PP\mbox{ and }q\in \QQ\}.
\]
Under the Minkowski sum $\mfp(n)$ becomes an abelian monoid, where the
identity element is the polytope consisting of the origin.  We denote
by $\mfg(n)$ the Grothendieck group of the monoid $\mfp(n)$. (See
Section~\ref{section:polytope-group} for details.)

We introduce a few more definitions:
\begin{enumerate}
\item The \emph{mirror image} of a polytope $\PP$ is
  $\ol{\PP}:=\{ -x\mid x\in \PP\}$.
\item A polytope $\PP$ is \emph{symmetric} if $\PP=\ol{\PP}$.
  Symmetric polytopes form a submonoid $\mfpsym(n)\subset \mfp(n)$ and
  a subgroup $\mfgsym(n)\subset \mfg(n)$.
\item An (integral) polytope $\PP$ is an (\emph{integral}) \emph{norm}
  if there exists an (integral) polytope $\QQ$ such that
  $\PP=\QQ+\ol{\QQ}$. (What we call norms are often referred to as
  {difference bodies}, but in light of
  Section~\ref{subsection:motivation} we prefer the non-standard name
  of a {norm}.)  Integral norms form a submonoid $\mfpnorm(n)\subset
  \mfp(n)$ and they generate a subgroup $\mfgnorm(n)\subset \mfg(n)$.
\end{enumerate}

Clearly a polytope that is a norm is also symmetric.  In the real
setting the converse holds. More precisely, any symmetric polytope
$\PP$ can be written as
\[ \PP\,=\, \tmfrac{1}{2}\PP+\tmfrac{1}{2}\PP\,=\, \tmfrac{1}{2}\PP+\ol{\tmfrac{1}{2}\PP}.\]
This shows that symmetric  polytopes are also  norms. 

In the remainder of the paper we study only integral polytopes and integral norms. Since any integral norm is symmetric it follows that $\mfpnorm(n)\subset
\mfpsym(n)$ and $\mfgnorm(n)\subset \mfgsym(n)$ for any~$n$.
We address the question whether all symmetric integral polytopes are integral norms.
The question arises naturally on its own, and in addition, there is a
motivation from the study of group rings and low dimensional topology.
See Section~\ref{subsection:motivation} for a related discussion.

Every one-dimensional symmetric integral polytope $\PP$ is of the form
$\PP=[-x,x]$ for some $x\in \Z_{\geq 0}$.  It can be written as
$\PP=\QQ+\ol{\QQ}$ where $\QQ=[0,x]$.  This shows that every
one-dimensional symmetric integral polytope is in fact an integral norm. Thus
$\mfpnorm(1)=\mfpsym(1)$ and $\mfgnorm(1)=\mfgsym(1)$.

The situation is more subtle in  dimension  two and higher. First of all
we have the following elementary lemma.

\begin{lemma}\label{lem:notanorm}
  For any $n\ge 2$ we have $\mfpsym(n)\ne \mfpnorm(n)$.
\end{lemma}

Our main result is, that at least to our surprise, the situation is
very different if one considers the Grothendieck group. More
precisely, we have the following theorem.

\begin{theorem}\label{mainthm}
For any $n$ we have  $\mfgsym(n)=\mfgnorm(n)$.
\end{theorem}

\subsection*{Acknowledgments} 

The first author was supported by NRF grants 2013067043 and
2013053914.  The second author was supported by the SFB 1085 `Higher
invariants', funded by the Deutsche Forschungsgemeinschaft (DFG).
 The third author was supported by the GRK 1150 `Homotopy and Cohomology' funded by the DFG, the Max Planck Institute for Mathematics, and the Deutsche Telekom Stiftung.
We wish to thank Matthias Nagel for helpful conversations.

\section{Preliminaries}

\subsection{The polytope group}\label{section:polytope-group}

Let $n\in \N$.  It is straightforward to show that the monoid
$\mfp(n)$ of integral polytopes has the cancellation property, i.e.\
for polytopes $\PP,\QQ,\RR\in \mfp(n)$ with $\PP+\QQ=\PP+\RR$ we have
$\QQ=\RR$.  (For instance see~\cite[Lemma~3.1.8]{Sc93}.)

On $\mfp(n)\times\mfp(n)$, define $(\PP,\QQ)\sim (\PP',\QQ')$ if
$\PP+\QQ'=\PP'+\QQ$.  This is an equivalence relation since $\mfp(n)$
has the cancellation property.  Let $\mfg(n)$ be the set of
equivalence classes.  It is straightforward to see that $\mfg(n)$ is
an abelian group under 
\[ (\PP,\QQ)+(\PP',\QQ') := (\PP+\PP',\QQ+\QQ').\]
It is referred to as the \emph{Grothendieck group} of $\mfp(n)$.  It
is also straightforward to see that the map
\[
  \begin{array}{c@{}c@{}c}
  \mfp(n)&{}\longrightarrow{}& \mfg(n) \\
  \PP&{}\longmapsto{}&(\PP,0)
  \end{array}
\]
is a monomorphism. We will use this monomorphism to identify $\mfp(n)$
with its image in~$\mfg(n)$. As usual, given $\PP$ and
$\QQ\in \mfp(n)$ we write $\PP-\QQ=(\PP,\QQ)$.

\subsection{Motivation: the marked polytope for a group ring element}
\label{subsection:motivation}

Here we discuss a motivation which leads us to consider
\emph{integral} polytopes.  Let $G$ be a finitely generated group.  An
\emph{integral polytope in $H_1(G;\R)$} is the convex hull of a finite number
of points in $\op{Im}\{H_1(G;\Z)\to H_1(G;\R)\}$. All the concepts and
definitions for polytopes in $\R^n$ generalize in an obvious way to
polytopes in $H_1(G;\R)$. In particular we can consider the monoid
$\mfp(H_1(G;\R))$ of integral polytopes in $H_1(G;\R)$ and we can consider the
corresponding group $\mfg(H_1(G;\R))$.

We denote by $\eps\co G\to H_1(G;\R)$ the canonical map.  Given a
non-zero element
\[
  f=\sum_{g\in G}f_gg\in \Z[G] \quad (f_g \in \Z)
\]
we refer to
\[
  \PP(f):=\text{convex hull of }\{\eps(g)\,|\, g\in G\mbox{ with }f_g\ne
    0\}\subset H_1(G;\R)
\]
as the \emph{polytope of $f$}. Now suppose that the ring $\Z[G]$ is a
domain, i.e.\ it has no non-zero element which is a left or right
zero-divisor. Conjecturally this is precisely the case when $G$ is
torsion-free. It is straightforward to see, that in this case the map
\[
  \begin{array}{c@{}c@{}c@{}c}
    \PP\colon{} & \Z[G]\sm \{0\} & {}\longrightarrow{} & \mfp(H)\\
                & f &{}\longmapsto{} & \PP(f)
  \end{array}
\]
is a monoid homomorphism. We refer to \cite[Lemma~3.2]{FT15} for
details.

Now let $G$ be a group that is torsion-free elementary amenable.  It
follows from \cite[Theorem~1.4]{KLM88} that the group ring $\Z[G]$ is
a domain.  Furthermore by \cite[Corollary~6.3]{DLMSY03} the ring
$\Z[G]$ satisfies the Ore condition, that is, for any two non-zero
elements $x,y\in \Z[G]$ there exist non-zero elements $p,q\in \Z[G]$
such that $xp=yq$. This implies that $\Z[G]$ admits a `naive' ring of
fractions, which usually is referred to as the Ore localization of
$\Z[G]$.  We refer to \cite[Section~4.4]{Pa77} for details.  In the
following we denote by $\K(G)^\times_{\op{ab}}$ the abelianization of
the multiplicative group $\K(G)^\times=\K(G)\sm \{0\}$.

The inversion $g\mapsto g^{-1}$ on $G$ extends linearly to the
standard involution $f\mapsto \overline{f}$ on~$\Z[G]$.  It extends
naturally to an involution on $\K(G)$ and on $\K(G)^\times_{\op{ab}}$.
A \emph{norm} in $\K(G)^\times_{\op{ab}}$ is an element that can be
written as $f\cdot \ol{f}$ for some $f\in \K(G)$.

If $G$ is an abelian group, then the question of whether elements in
$\K(G)^\times_{\op{ab}}$ are norms arises naturally in the study of
Alexander polynomials of link concordance and homology cylinders. We
refer to \cite{Ka78,Na78,Tu86,CFK11,CF13} for details. To a given link
or homology cylinder one can also associate `non-commutative Alexander
polynomials', that are elements in $\K(G)^\times_{\op{ab}}$, for
appropriate choices of non-abelian torsion-free elementary amenable
groups $G$. We refer to \cite{COT03,Co04,Ha05,Fr07,FH07,FV10} for
details.  Again, in the study of link concordance and homology
cylinders the question arises whether or not a given element in
$\K(G)^\times_{\op{ab}}$ is a norm, see e.g.\ \cite{Ki12}.

The group $\K(G)^\times_{\op{ab}}$ is unwieldy and little
understood. In particular it is difficult to determine whether or not
a given element is a norm. It is straightforward to see that the above
map $\PP\colon \Z[G]\sm \{0\} \to \PP(H)$ extends to a group
homomorphism
\[
  \PP\co \K(G)^\times_{\op{ab}} \longrightarrow \mfg(H).
\]
Furthermore this group homomorphism sends norms to norms.  Thus the
question arises which elements in $\mfg(H)$ are norms.  It is a
consequence of Poincar\'e duality, see e.g.\ \cite{FKK12}, that the
aforementioned 3-manifold invariants in $\K(G)^\times_{\op{ab}}$ are symmetric. In particular the corresponding elements in $\mfg(H)$ are symmetric. Thus the question
arises, whether there exist symmetric elements in $\mfg(H)$ that are
not norms.

\section{Proofs}

\subsection{Proof of Lemma~\ref{lem:notanorm}}
Let $\PP$ be a polytope in $\R^n$.  Recall that a face of $\PP$ is a
polytope in its own right. In fact a face of $\PP$ is the convex hull of a
proper subset of the vertex set of~$\PP$.  We call a polytope
contained in a face a \emph{subface of~$\PP$}. We leave the proof of
the following elementary lemma as an exercise to the reader.

\begin{lemma}\label{lem:subface}
  Let $\PP$ and $\QQ$ be polytopes in $\R^n$. Then any face of $\PP$
  is, up to translation, a subface of $\PP+\QQ$.
\end{lemma}

Now we are in a position to prove Lemma~\ref{lem:notanorm}.

\begin{proof}[Proof of Lemma~\ref{lem:notanorm}]
  We first show that $\mfpsym(2)\ne \mfpnorm(2)$.  Let $k\in \N$. We
  denote by $\PP$ the integral 2-dimensional symmetric polytope
  spanned by $(k,0), (k,1), (-k,0)$ and $(-k,-1)$. We want to show
  that $\PP$ is not an integral norm. We denote by $\XX$ the integral
  polytope spanned by $(0,0)$ and $(2k,1)$ and we denote by $\YY$ the
  integral polytope spanned by $(0,0)$ and $(0,1)$.

  Suppose $\PP$ is an integral norm. Thus we can write
  $\PP=\QQ+\ol{\QQ}$, where $\QQ$ is an integral polytope. By
  Lemma~\ref{lem:subface}, each face of $\QQ$ is, up to translation, a
  subface of $\PP$. This implies that, up to translation, each face of
  $\QQ$ is a subface of $\XX$ or of $\YY$. Since neither $\XX$ or
  $\YY$ admits one-dimensional integral subpolytopes we see that, up
  to translation, each face of $\QQ$ is either $\XX$, $\YY$ or a
  point. In particular, up to translation, $\QQ$ equals either
  $\{0\}$, $\XX$, $\YY$ or $\XX+\YY$. But it is straightforward to
  verify that in each case $\QQ+\ol{\QQ}\ne \PP$.
  
  This shows that $\mfpsym(2)\ne \mfpnorm(2)$. Now let $n\geq 3$. We
  consider the two maps
  \[
  \begin{array}{r@{}c@{}l} \Phi\colon \R^2&{}\longrightarrow{}& \R^n\\
    (x,y)&{}\longmapsto{}&(x,y,0)\end{array} \quad \mbox{ and } \quad
  \begin{array}{r@{}c@{}l} \Psi\colon \R^n&{}\longrightarrow{}& \R^2\\
    (x_1,\dots,x_n)&{}\longmapsto{}&(x_1,x_2)\end{array}.
  \]
  Both maps induce homomorphisms on the polytope monoids that map
  symmetric polytopes to symmetric polytopes and integral norms to
  integral norms.  Clearly $\Psi$ is a splitting of $\Phi$. Now it
  follows that if $\PP$ is an integral symmetric polytope in $\R^2$
  that is not a norm, then $\Phi(\PP)\subset \R^n$ is also a symmetric
  polytope that is not an integral norm.
\end{proof}

\subsection{Proof of Theorem~\ref{mainthm}}
In this section we will prove Theorem~\ref{mainthm}.  So given any $n$
we want to show that $\mfgsym(n)=\mfgnorm(n)$. This is equivalent to
showing that given any integral polytope $\PP$ in $\R^n$ there exist
integral polytopes $\QQ$ and $\RR$ such that
\begin{equation}
  \label{equation:norm-decomposition}
  \PP + \QQ + \overline{\QQ} = \RR + \overline{\RR}. \tag{$*$}
\end{equation}
The key idea is to prove this statement by induction on $n$ where we
perform the induction step by cutting along a hyperplane.

In the proof of Theorem~\ref{mainthm} we will use the following
definitions and notations.  A hyperplane $H$ in $\R^n$ can be written
as $H=\{ x\in \R^n \mid x\cdot v = 0\}$ for some $v\in \R^n$.  We
define the \emph{halves} of a real polytope $\PP \subset \R^n$ with
respect $H$ to be
\begin{align*}
  \PP_+ &:=\{x\in \PP \mid x\cdot v \ge 0\},\\
  \PP_- &:=\{x\in \PP \mid x\cdot v \le 0\}.
\end{align*}
Informally speaking, when $H$ meets $\PP$ in a proper subset, $\PP_+$ and
$\PP_-$ are obtained by cutting $\PP$ along~$H$.  We remark that
$\PP_+$ and $\PP_-$ may be exchanged depending on the choice of $v$,
but it will not cause any issue for our purpose.  It is known that
each of $\PP_+$, $\PP_-$ and $\PP\cap H$ is a real polytope whenever
it is nonempty.

\begin{lemma}[Normalization by a hyperplane]
  \label{lemma:division-by-hyperplane}
  Suppose $\PP\in\R^n$ is a symmetric polytope and
  $H\subset \R^n$ is a hyperplane.  Let $\PP_+$ and $\PP_-$ be the
  halves of $\PP$ with respect to~$H$.  Then
  \[
    \PP_++\ol{\PP_+} = \PP_-+\ol{\PP_-} = \PP_++\PP_- = \PP+(\PP \cap
    H).
  \]
\end{lemma}

\begin{proof}
  Since $\PP$ is symmetric, $\ol{\PP_{\pm}} = \PP_{\mp}$.  Therefore
  it suffices to show that $\PP_++\PP_- = \PP+(\PP \cap H)$.
  
  Each $x\in \PP$ lies in either $\PP_+$ or~$\PP_-$.  If $x\in \PP_+$,
  then since $\PP\cap H = \PP_+\cap \PP_- \subset \PP_-$, we have
  $\{x\}+(\PP\cap H) \subset \PP_+ + \PP_-$.  By symmetry,
  $\{x\}+(\PP\cap H) \subset \PP_+ + \PP_-$ when $x\in \PP_-$.  It
  follows that $\PP+(\PP \cap H) \subset \PP_++\PP_-$.

  For the reverse inclusion, suppose $x\in \PP_+$ and $y\in \PP_-$.
  Since $\PP$ is convex, there is $t\in [0,1]$ such that the point
  $z:=tx+(1-t)y$ lies on~$\PP\cap H$.  Consider $p:=(1-t)x+ty$.  Since
  $\PP$ is convex, $p\in \PP$.  Therefore $x+y = p+z$ lies in
  $\PP+(\PP\cap H)$.  This shows
  $\PP_++\PP_- \subset \PP+(\PP \cap H)$.
\end{proof}

We can not directly apply Lemma~\ref{lemma:division-by-hyperplane} to an integral polytope, since in general, given an integral polytope $\PP$ there does not exist a hyperplane, such that $\PP\cap H$ is again integral.
  To overcome
this, the following is useful:

\begin{lemma}[Vertical stretching]
  \label{lemma:vertical-stretch}
  Let $\PP$ be an integral polytope in $\R^n$.  Denote by $d\ZZ$ the line
  segment in $\R^n$ joining the origin and the point~$(0,\dots,0,d)$.
  As usual, identify $\R^{n-1}$ with the hyperplane of points with
  last coordinate zero in~$\R^n$.  Then for all sufficiently large
  $d>0$,
  \[
  (\PP +d\ZZ+\ol{d\ZZ})\cap \R^{n-1}
  \]
  is integral.
\end{lemma}

\begin{proof}
  Denote by $\pi\colon \R^n\to \R^{n-1}$ the projection map which
  forgets the last coordinate.  Denote by $v_1,\dots,v_k\in \Z^n$ the
  vertices of~$\PP$.  Let $w_i=\pi(v_i)\in \R^{n-1}$ and write
  $v_i=(w_i,a_i)$ with $a_i\in \Z$.  Let $\YY$ be the convex hull of
  $\{w_1,\ldots,w_k\}$.  Suppose $d$ satisfies $d>|a_i|$ for all~$i$.
  Now it suffices to prove the following:

  \begin{claim}
    $(\PP + d\ZZ+\ol{d\ZZ})\cap \R^{n-1}=\YY$.
  \end{claim}

  Obviously we have
  \[
  (\PP +d\ZZ+\ol{d\ZZ})\cap \R^{n-1}\, =\, \pi\big((\PP +d\ZZ+\ol{d\ZZ})\cap
  \R^{n-1}\big)\, \subset \,\pi(\PP +d\ZZ+\ol{d\ZZ}).
  \]
  Since $\pi(d\ZZ)=\{0\}$, $\pi(\PP +d\ZZ+\ol{d\ZZ})=\pi(\PP)$.  Since
  the projection of the convex hull of a set is the convex hull of the
  projection of the set, we have $\pi(\PP)=\YY$.  It follows that
  $(\PP +d\ZZ+\ol{d\ZZ})\cap \R^{n-1}\subset \YY$.

  For the reverse inclusion, observe that $(w_i,a_i\pm d)\in \PP
  +d\ZZ+\ol{d\ZZ}$ for each~$i$.  Note that one of $a_i\pm d$ is
  negative and the other is positive.  By convexity we have
  $(w_i,0)\in \PP +d\ZZ+\ol{d\ZZ}$.  Once again by convexity, it
  follows that
  \[
  \YY \subset (\PP + d\ZZ+\ol{d\ZZ})\cap \R^{n-1}.  \qedhere
  \]
\end{proof}

Now we are ready for the proof of the main result of this paper.

\begin{proof}[Proof of Theorem~\ref{mainthm}]
  We prove the theorem by induction on $n$.  For $n=0$, the statement
  is trivial.  Suppose the conclusion holds for $n-1$, and suppose
  $\PP$ is a symmetric integral polytope in~$\R^n$.  As above we
  identify $\R^{n-1}$ with the hyperplane of points with last
  coordinate zero in~$\R^n$.  By Lemma~\ref{lemma:vertical-stretch},
  there is $d\in\N$ such that $(\PP +d\ZZ+\ol{d\ZZ})\cap \R^{n-1}$ is
  an integral polytope in~$\R^{n-1}$.  Write $\YY=\PP +d\ZZ+\ol{d\ZZ}$
  for brevity.  Since $\YY$ is symmetric, $\YY\cap\R^{n-1}$ is
  symmetric too.  Therefore, by the induction hypothesis, there are
  integral polytopes $\QQ$ and $\RR$ in $\R^{n-1}$ such that
  \[
  (\YY \cap \R^{n-1}) + \QQ+\ol{\QQ} = \RR+\ol{\RR}.
  \]
  By Lemma~\ref{lemma:division-by-hyperplane}, we have
  \[
  \YY + (\YY\cap \R^{n-1}) = \YY_+ +\ol{\YY_+}
  \]
  where $\YY_+$ denotes a half of $\YY$ with respect to the
  hyperplane~$\R^{n-1}$. Since $\YY\cap \R^{n-1}$ is integral we also deduce that $\YY_+$ is an integral polytope. From the above equations, it follows that
  \[
  \PP + (d\ZZ + \QQ) + \ol{(d\ZZ + \QQ)} = (\YY_+ + \RR) +\ol{(\YY_+ +
    \RR)}. \qedhere
  \]
\end{proof}

\end{document}